\documentclass[10pt]{amsart}

\usepackage{amssymb}
\usepackage[all]{xy}
\usepackage{enumerate}

\newcommand{\Z}{\ensuremath{\mathbb{Z}}}
\newcommand{\R}{\ensuremath{\mathbb{R}}}
\newcommand{\C}{\ensuremath{\mathbb{C}}}
\newcommand{\Ri}{\ensuremath{\mathbb{R}_{\infty}}}
\newcommand{\Zi}{\ensuremath{\mathbb{Z}_{\infty}}}
\newcommand{\fr}[3]{\ensuremath{{#1}_{#2}\langle{#3}\rangle}}
\newcommand{\ff}[3]{\ensuremath{{#1}_{#2}(\langle{#3}\rangle)}}

\DeclareMathOperator{\Aut}{Aut}

\DeclareMathOperator{\op}{op}
\DeclareMathOperator{\supp}{supp}
\DeclareMathOperator{\charac}{char}

\theoremstyle{plain}
\newtheorem{theorem}{Theorem}[section]
\newtheorem{proposition}[theorem]{Proposition}
\newtheorem{lemma}[theorem]{Lemma}
\newtheorem{corollary}[theorem]{Corollary}

\theoremstyle{definition}
\newtheorem{example}[theorem]{Example}

\begin{document}

\title{Free fields in Malcev-Neumann series rings}

\author[V. O. Ferreira]{Vitor O. Ferreira}
\address{Department of Mathematics - IME, University of S\~ao Paulo,
Caixa Postal 66281, S\~ao Paulo, SP, 05314-970, Brazil}
\email{vofer@ime.usp.br}
\thanks{The first author is the corresponding author and was partially
supported by CNPq, Brazil (Grant 308163/2007-9)}

\author[E. Z. Fornaroli]{\'Erica Z. Fornaroli}
\address{Department of Mathematics, State University of Maring\'a,
Avenida Colombo, 5790, Maring\'a, PR, 87020-900, Brazil}%
\email{ezancanella@uem.br}
\thanks{The second author was supported by CNPq, Brazil (Grant 141505/2005-2)}

\author[J. S\'{a}nchez]{Javier S\'{a}nchez}
\address{Department of Mathematics - IME, University of S\~ao Paulo,
Caixa Postal 66281, S\~ao Paulo, SP, 05314-970, Brazil}%
\email{jsanchez@ime.usp.br}
\thanks{The third author was partially supported by FAPESP, Brazil (Proc.~2009/50886-0),
by the DGI and the European Regional Development
Fund, jointly, through Project MTM2008–06201–C02–01, and by the Comissionat per
Universitats i Recerca of the Generalitat de Catalunya, Project 2009 SGR 1389.}



\date{16 August 2011}

\begin{abstract}
It is shown that the skew field of Malcev-Neumann series of an ordered
group frequently contains a free field of countable rank, \textit{i.e.}~the universal
field of fractions of a free associative algebra of countable rank. This is an application
of a criterion on embeddability of free fields on skew fields which are
complete with respect to a valuation function, following K.~Chiba.
Other applications to skew Laurent
series rings are discussed. Finally, embeddability questions on free fields
of uncountable rank in Malcev-Neumann series rings are also considered.
\end{abstract}

\maketitle

\section*{Introduction}

The existence of free objects in skew fields has attracted considerable
attention since Lichtman \cite{aL77} conjectured that the multiplicative group
of a noncommutative skew field contains noncyclic free subgroups.
A related problem, regarding free subalgebras inside skew fields,
has also been the subject of investigation although somewhat less intensely.
And free group algebras have been
conjectured to exist in finitely generated skew fields which are infinite-dimensional
over their centres. We refer to the survey article \cite{GSpp} and the references therein.

More recently, K.~Chiba \cite{kC03} has considered complete skew fields
with respect to a valuation function and proved that they contain free fields
provided they are infinite-dimensional over their centres. Moreover, G.~Elek \cite{gE06}
has considered free subfields of skew fields in the context of amenability.

In this paper, we present some improvements on Chiba's results which
yield conditions under which skew fields of Malcev-Neumann series
of ordered groups are shown to contain free fields of countable rank. These valuation
methods are later also applied to skew Laurent series rings.

A remark  of A.~Lichtman on the nonexistence of free fields in the field
of fractions of group algebras of torsion free polycyclic-by-finite groups
is presented in Section 4, showing that the free algebras that are known
to exist in these skew fields do not generate free fields.

A final section is devoted to the generalization of the valuation methods
used earlier in order to guarantee the presence of free fields of
uncountable rank in some Malcev-Neumann series rings.

\section{Definitions and notation}

In what follows the terms ``division ring'' and ``skew field'' will
be used interchangeably. Occasionally, we shall omit the adjective
``skew''. A skew field which is commutative will always be called a
``commutative field''.

Given a skew field $D$, a subfield $K$ of $D$ and a set $X$,
the \emph{free $D_K$-ring} on $X$ is defined to be the ring
$\fr{D}{K}{X}$ generated by $X$ over $D$ satisfying the relations
$ax=xa$ for all $a\in K$ and $x\in X$. It is well known that
$\fr{D}{K}{X}$ is a fir and, thus, has a universal field of
fractions $\ff{D}{K}{X}$, called the \emph{free $D_K$-field} on $X$.
(We refer to \cite{pC85} for the theory of firs and their universal
fields of fractions.) The free $D_D$-ring on $X$ will be called the
free $D$-ring on $X$ and will be denoted by $\fr{D}{}{X}$. Its
universal field of fractions, the free $D$-field, will be denoted by
$\ff{D}{}{X}$.

We shall be looking at valuation functions on skew fields. Let $R$
be a ring and let $(G,<)$ be a (not necessarily abelian) ordered group
with operation denoted additively. A symbol $\infty$ is adjoined
to $G$ and in $G_{\infty} = G\cup\{\infty\}$ the operation and order
are extended in such a way that $x+\infty=\infty+x=\infty+\infty=\infty$
and $x<\infty$, for all $x\in G$. By a \emph{valuation on $R$ with
values in $G$} one understands a map $\nu\colon R\rightarrow
G_{\infty}$ satisfying
\renewcommand{\theenumi}{\roman{enumi}}
\begin{enumerate}
  \item $\nu(a)=\infty$ if and only if $a=0$,
  \item $\nu (ab)=\nu(a)+ \nu(b)$, for all $a,b \in R$,\label{val1}
  \item $\nu(a+b)\geq \min\{\nu(a), \ \nu(b)\}$, for all $a,b \in R$.
\end{enumerate}
The valuation $\nu$ is said to be \emph{trivial} if $\nu(a)=0$, for all $a\in R, a\ne 0$.
When $R$ is a skew field and $R^{\times}$ denotes the set of all nonzero elements
of $R$, $\nu(R^{\times})$ is a subgroup of $G$, called the \emph{value group} of $\nu$.
A valuation on a skew field with value group order-isomorphic to the ordered
group of the integer numbers is said to be \emph{discrete}.
\renewcommand{\theenumi}{\arabic{enumi}}

We shall be mainly concerned with valuations on skew fields with
values in $\R$, the ordered additive group of the real
numbers with its usual order. If $D$ is a skew field and
$\nu\colon D\rightarrow\Ri$ is a valuation, we can define a metric $d$ on $D$
by choosing a real constant $c\in (0, 1)$ and letting
$d(a,b)=c^{\nu(a-b)}$, for all $a,b\in D$.
The topology so defined on $D$ is independent of the choice of the
constant $c$, and the completion $\widehat{D}$ of the metric space
$D$ is again a skew field with a valuation $\hat{\nu}\colon\widehat{D}\rightarrow\Ri$
such that $\widehat{D}$ and $\hat{\nu}$ are extensions of $D$ and $\nu$,
respectively. We shall refer to $\widehat{D}$ as the
\emph{completion of $D$ with respect to the valuation $\nu$}.

An example is the following. Let $R$ be a ring with a central
regular element $t$ such that $\bigcap{t^nR}=0$ and $R/tR$ is a
domain. If we define, for each $a\in R$, $\nu(a)=\sup\{n : a\in
t^nR\}$ then $\nu$ is a valuation on $R$, called the \emph{$t$-adic
valuation}. In particular, if $R$ is a domain and if $R[z]$ is the
polynomial ring over $R$ on the indeterminate $z$, $R[z]$ has a
$z$-adic valuation. If, moreover, $R$ is a division ring, then
$R[z]$ is an Ore domain with field of fractions $R(z)$, the field of
rational functions on $z$. The $z$-adic valuation on $R[z]$ extends
to a discrete valuation $\nu$ on $R(z)$ and the completion of $R(z)$
with respect to $\nu$ is just the field of Laurent series $R((z))$.

\medskip

The following further notions regarding ordered groups will be needed in
the sequel. Details can be found in \cite{BMR77} or \cite{lF63}, for example.

Given an ordered group $(G,<)$, a subgroup $H$ of $G$ is said to
be \emph{convex} if for all $x,z\in H$ and $y\in G$, the inequalities
$x\leq y\leq z$ imply that $y\in H$. If a convex subgroup $N$ is normal,
then $G/N$ has a natural order inherited from that on $G$, and the
convex subgroups of $G/N$ are in one-to-one correspondence with those
of $G$ which contain $N$.

The ordered group $(G,<)$ is said to be \emph{archimedean} if $\{0\}$
and $G$ are the only convex subgroups of $G$. An archimedean group
is known to be order isomorphic to a subgroup of the additive group $\R$
of the real numbers with its natural order.

The set of convex subgroups of an ordered group $(G,<)$ is totally
ordered with respect to inclusion, and we say that a pair $(N,H)$ of
convex subgroups of $G$ is a \emph{convex jump} if $N\subsetneqq H$ and there
are no convex subgroups lying properly between $N$ and $H$. Given
a convex jump $(N,H)$, $N$ is known to be normal in $H$ and, therefore, $H/N$ is
order isomorphic to a nontrivial subgroup of $\R$.

\medskip

Finally, we recall that given a skew field $K$ and an ordered group $(G,<)$,
the \emph{Malcev-Neumann series ring} $K((G,<))$ is defined to be the set
of all formal series $f=\sum_{x\in G}xa_x$, with $a_x\in K$, whose
\emph{support} $\supp(f)=\{x\in G : a_x\ne 0\}$ is a well-ordered subset
of $G$. The operations of $K$ and $G$, together with the relations $ax=xa$, for
all $a\in K$ and $x\in G$ induce a multiplication in $K((G,<))$, with respect
to which $K((G,<))$ is a field. Clearly, $K((G,<))$ contains the
group ring $K[G]$ of $G$ over $K$. We denote by $K(G)$ the subfield of $K((G,<))$
generated by $K[G]$. In Section~\ref{sec:mn}, we shall encounter a more general construction.

\renewcommand{\theenumi}{\arabic{enumi}}

\section{Free fields in valued division rings}\label{sec:gendr}

We start by remarking that Chiba's main result in \cite{kC03}
holds in a more general setting, namely, we can consider
skew fields having a valuation with noninteger values. This
remark will be used in the following section in order to
produce embeddings of free fields in Malcev-Neumann skew
fields of arbitrary ordered groups.

\medskip

The proof of the following result can be extracted from the proof of
\cite[Theorem~1]{kC03}.

\begin{theorem}\label{th:chiba}
  Let $D$ be a skew field with infinite centre and let $K$ be a subfield of
  $D$ which is its own bicentralizer and whose centralizer $K'$ in $D$ is
  such that the left $K$-space $KcK'$ is infinite-dimensional,
  for all $c\in D^{\times}$. Suppose that $\nu\colon D\rightarrow \Ri$ is a valuation
  satisfying the condition that there exists a nonzero element $t$ of $K'$ with $\nu(t)>0$.
  Then for every infinite countable set $\Sigma$ of full matrices over the
  free $D_K$-ring $\fr{D}{K}{X}$ there exists a
  $\Sigma$-inverting homomorphism from $\fr{D}{K}{X}$ into
  the completion $\widehat{D}$ of $D$ with respect to the valuation $\nu$.\hfill\qed
\end{theorem}

(In particular, it follows that when $D$ and $X$ are infinite countable,
$\nu$ is a discrete valuation and $\Sigma$ is the set of all full
matrices over $\fr{D}{K}{X}$ the completion $\widehat{D}$ contains
the free field $\ff{D}{K}{X}$. That is the statement of
\cite[Theorem~1(1)]{kC03}.)

The following consequence of Theorem~\ref{th:chiba} should be
compared with \cite[Corollary~1(1)]{kC03}.

\begin{theorem}\label{th:freefield}
  Let $D$ be a skew field with infinite centre $C$ such that
  the dimension of $D$ over $C$ is infinite. If there exists a
  nontrivial valuation $\nu\colon D\rightarrow\Ri$,
  then $\widehat D$ contains a free
  field $\ff{C}{}{X}$ on an infinite countable set $X$.
\end{theorem}

\begin{proof}
  Let $X$ be an infinite countable set and let $F$ be an infinite countable subfield of $C$.
  Let $\Sigma$ denote the set of all full matrices over $\fr{F}{}{X}$.
  Since both $F$ and $X$ are infinite countable, $\Sigma$ is infinite countable. Moreover,
  $\Sigma$ is a set of full matrices over $\fr{D}{C}{X}$ because the natural
  inclusion $\fr{F}{}{X} \hookrightarrow \fr{D}{C}{X}$
  is an honest map, by \cite[Theorem~6.4.6]{pC95}. It follows
  from Theorem~\ref{th:chiba} that there exists a $\Sigma$-inverting
  homomorphism $\fr{D}{C}{X} \rightarrow \widehat{D}$ and, therefore,
  the composed map
  $$
    \fr{F}{}{X}\hookrightarrow\fr{D}{C}{X}\longrightarrow\widehat{D}
  $$
  is a $\Sigma$-inverting $F$-ring homomorphism which extends to an $F$-ring homomorphism
  $\ff{F}{}{X}\rightarrow\widehat{D}$. Thus $X$ freely generates a free
  subfield of $\widehat{D}$ over $F$. Since $F$ is a central subfield of $D$,
  $X$ freely generates a free subfield of $\widehat{D}$ over the prime field of $D$ and,
  \textit{a fortiori}, over $C$, by \cite[Lemma~9]{kC03}.
\end{proof}

This provides a more direct proof of the following version of \cite[Corollary~1(2)]{kC03}.

\begin{corollary}\label{cor:Laurentseries}
  Let $K$ be a skew field with centre $F$ such that the
  dimension of $K$ over $F$ is infinite. Then the skew field of
  Laurent series $K((z))$ in $z$ over $K$ contains a free field
  $\ff{F}{}{X}$ on an infinite countable set $X$.
\end{corollary}

\begin{proof}
  Let $\omega$ denote the valuation on $K(z)$ which extends the $z$-adic
  valuation on $K[z]$. Since there is a natural embedding
  $K\otimes_F F(z)\hookrightarrow K(z)$, it follows that $[K(z):F(z)]\geq
  [K\otimes_FF(z):F(z)]=[K:F]$ and, hence, $K(z)$ is infinite dimensional over $F(z)$, which is
  its (infinite) centre, by \cite[Proposition~2.1.5]{pC95}. So Theorem~\ref{th:freefield}
  applies and, therefore, $K((z))$ contains a free field $\ff{F(z)}{}{X}$
  on an infinite countable set $X$. Finally, by \cite[Lemma~9]{kC03}, $K((z))$
  contains a free field $\ff{F}{}{X}$.
\end{proof}

Now let $K$ be a skew field and let $(G,<)$ be a nontrivial ordered group
with operation denoted multiplicatively.
Given an element $z\in G$ with $z>1$, let $H$ denote the subgroup of $G$
generated by $z$. The subfield $K((H,<))$ of $K((G,<))$ is isomorphic to the Laurent
series skew field $K((z))$. So, if the dimension of $K$ over its centre $F$
is infinite, Corollary~\ref{cor:Laurentseries} provides an embedding
of the free field $\ff{F}{}{X}$ into $K((H,<))$ and, therefore, into
$K((G,<))$. This proves the next result.

\begin{corollary}\label{cor:MNbig}
  Let $K$ be a skew field with centre $F$ such that the dimension of $K$
  over $F$ is infinite and let $(G,<)$ be a nontrivial ordered group.
  Then the Malcev-Neumann skew field $K((G,<))$ contains a free field
  $\ff{F}{}{X}$ on an infinite countable set $X$.\hfill\qed
\end{corollary}

In the next section we shall see that most Malcev-Neumann skew fields
contain free fields, even when the field of coefficients is not large
compared with its centre.

\section{Free fields in Malcev-Neumann series rings}\label{sec:mn}

In this section we consider arbitrary Malcev-Neumann skew fields
containing crossed product rings. These fields will be shown to
contain free fields very frequently. In particular, we shall
see that Malcev-Neumann series over torsion-free nilpotent
groups will have this property. Finally, some of the Malcev-Neumann skew fields
considered will also be shown to be nonamenable in the sense
of \cite{gE06}.

First we recall the construction of rings of Malcev-Neumann
series containing crossed products. We follow \cite[Chapter~1]{dP89}.

Given a skew field $K$ and a group $G$, a \emph{crossed product}
of $G$ over $K$ is an associative ring $K[G;\sigma,\tau]$
which contains $K$ and which is a free right $K$-module
with basis $\overline{G}=\{\bar{x} : x\in G\}$, a copy of $G$.
Thus, each element $f\in K[G;\sigma,\tau]$ is uniquely expressed as
$f=\sum_{x\in G}\bar{x}a_x$, where $a_x\in K$, for all $x\in G$, and
the support $\supp f=\{x\in G : a_x\neq 0\}$ of $f$ is finite.
Here, multiplication satisfies
$$
\bar{x}\bar{y}=\overline{xy}\tau(x,y)\quad\text{and}\quad
a\bar{x}=\bar{x}a^{\sigma(x)},
$$
for all $x,y\in G$ and $a\in K$, where $\tau\colon G\times G\rightarrow K^{\times}$
and $\sigma\colon G\rightarrow \Aut(K)$ are maps. We can always
assume that the identity element of $K[G;\sigma,\tau]$ is $\bar{1}$
and that $K$ is a subfield of $K[G;\sigma,\tau]$. If $H$ is a subgroup
of $G$, then the crossed product $K[H;\sigma_{\mid H},\tau_{\mid H\times H}]$ embeds
into $K[G;\sigma,\tau]$. This subring will be denoted simply by $K[H;\sigma,\tau]$.

Suppose that $(G,<)$ is an ordered group and that $K[G;\sigma,\tau]$ is
a crossed product of $G$ over a skew field $K$. Then we can define a new
ring, denoted $K((G;\sigma,\tau,<))$ and called \emph{Malcev-Neumann
series ring}, into which $K[G;\sigma,\tau]$ embeds. The elements of
$K((G;\sigma,\tau,<))$ are uniquely expressed as $f=\sum_{x\in
G}\bar{x}a_x$, with $a_x\in K$, for all $x\in G$, and $\supp f$ a
well-ordered subset of $G$. Addition and multiplication are defined
extending the ones in $K[G;\sigma,\tau]$. Thus, given
$f=\sum\limits_{x\in G}\bar{x}a_x$ and $g=\sum\limits_{x\in
G}\bar{x}b_x$ in $K((G;\sigma,\tau,<))$, addition and
multiplication are given by
$$
f+g = \sum_{x\in G}\bar{x} (a_x+b_x),\qquad f\cdot g = \sum_{x\in
G}\bar{x}\Big(\sum_{yz=x} \tau(y,z)a_y^{\sigma(z)} b_z\Big).
$$

It is a known fact that $K((G;\sigma,\tau,<))$ is a skew field
provided $K$ is (\cite{aM48,bN49}). If $H$
is a subgroup of $G$, the field $K((H;\sigma_{\mid H},\tau_{\mid H\times H},<))$
embeds into $K((G;\sigma,\tau,<))$ and will be denoted by $K((H;\sigma,\tau,<))$.

A few words are due on Malcev-Neumann series rings with regards to
the invariance of their construction. First, we remark that a
diagonal change of basis in $K[G;\sigma,\tau]$ gives rise to the
same series ring. Moreover, an order defined via an automorphism of
$G$ which is induced by an automorphism of $K[G;\sigma,\tau]$ will
produce a Malcev-Neumann series ring which is isomorphic to
$K((G;\sigma,\tau,<))$.

The \emph{field of fractions} of
$K[G;\sigma,\tau]$ inside $K((G;\sigma,\tau,<))$, \textit{i.e.}~the subfield
of $K((G;\sigma,\tau,<))$ generated as a skew field by
$K[G;\sigma,\tau]$, will be denoted by $K(G;\sigma,\tau)$. We
remark that $K(G;\sigma,\tau)$ does not depend on the order $<$ by
\cite[p.~183]{iH70}. When we have a group ring $K[G]$, we will
denote the Malcev-Neumann series ring by $K((G;<))$ and the corresponding
field of fractions of $K[G]$ by $K(G)$.

Now, let $K$ be a skew field and let $(G,<)$ be an ordered group
which contains a noncyclic free subgroup $H$ on a set $X$. Let
$K[G;\sigma,\tau]$ be a crossed product and let $P$ denote the prime
subfield of $K$. Because $H$ is a free group, after a diagonal change 
of basis, if necessary, it can be seen that the group algebra $P[H]$ is 
contained in $K[G;\sigma,\tau]$ and, correspondingly, that $P((H,<))$ is
contained in $K((G;\sigma,\tau,<))$. By a result of Lewin~\cite{jL74},
the field of fractions $P(H)$ of $P[H]$ inside $P((H,<))$ is isomorphic
to the free field $\ff{P}{}{X}$ --- so, in this case, we have an
embedding of a free field into $K(G,\sigma,\tau)\subseteq
K((G;\sigma,\tau,<))$. We shall see below that Malcev-Neumann series rings 
of a much larger class of ordered groups will have the property of 
containing free fields.

It is well known that the map $\omega\colon
K((G;\sigma,\tau,<))\rightarrow G_{\infty}$ given by
$\omega(0)=\infty$ and $\omega(f)=\min\supp f$, for all nonzero $f\in
K((G;\sigma,\tau,<))$, is a valuation of $K((G;\sigma,\tau,<))$ with
value group $G$. (Note that here we use multiplicative notation for the
operation on $G$ and we shall do so whenever considering this natural
valuation function.) If $(A,<')$ is another ordered group and $\pi\colon
(G,<)\rightarrow (A,<')$ is a surjective order homomorphism, then
$\nu\colon K((G;\sigma,\tau,<))\rightarrow A_{\infty}$, defined
by $\nu(0)=\infty$ and $\nu(f)=\pi\omega(f)$, for all nonzero $f\in
K((G;\sigma,\tau,<))$, is a valuation on $K((G;\sigma,\tau,<))$ with
value group $A$.

In order to prove the existence of a free field inside
$K((G;\sigma,\tau,<))$ using Theorem~\ref{th:freefield}, it is
enough to establish for some subgroup $H$ of $G$ the infinite
dimensionality of $K(H;\sigma,\tau)$ over its centre and that
$K((H;\sigma,\tau,<))$ contains the completion
$\widehat{K(H;\sigma,\tau)}$ of $K(H;\sigma,\tau)$ with respect to
some valuation function on $K(H;\sigma,\tau)$ with real values.
The former is established in the next lemma which
is proved in the same way as \cite[Lemma~10]{kC03}, the latter is
done in Lemma~\ref{le:completemns}.

\begin{lemma}\label{le:ffskg}
  Let $K$ be a skew field and let $(G,<)$ be a nonabelian ordered group.
  Then any field of fractions of a crossed product
  of $G$ over $K$ is infinite dimensional over its centre.
\end{lemma}

\begin{proof}
  Let $R=K[G;\sigma,\tau]$ be a crossed product of $G$ over $K$,
  let $D$ be a field of fractions of $R$ and let
  $C$ be the centre of $D$.

  We consider two cases. If $R$ is not an Ore
  domain, then, by \cite[Proposition~10.25]{tL99}, it contains a
  noncommutative free algebra
  over its centre. This implies that $D$ contains a
  free subalgebra over $C$, by \cite[Lemma~1]{MM91}, and, so,
  $[D:C]$ must be infinite.

  If $R$ is an Ore domain, then $D$ is the Ore
  field of fractions of $R$, and, hence $D=K(G;\sigma,\tau)
  \subseteq K((G;\sigma,\tau,<))$.

  Denote the centre of $G$ by $Z$. Now, since $R$ is an Ore domain, it
  is well-known that $K[Z;\sigma,\tau]$ must also be an Ore domain and, so,
  the Ore field of fractions $K(Z;\sigma,\tau)$
  of $K[Z;\sigma,\tau]$ is contained in $D$. We shall see, first, that
  $C\subseteq K((Z;\sigma,\tau,<))$.

  Indeed, let $f\in C$ and express $f$ as a series in $K((G;\sigma,\tau,<))$, say
  $f=\sum_{x\in G}\bar{x}a_x$. For each $y\in G$, we have $\bar{y}f=f\bar{y}$.
  Thus,
  $$
  \sum_{x\in G}\overline{yx}\tau(y,x)a_x=
  \bar{y}f=f\bar{y}=
  \sum_{x\in G}\overline{xy}\tau(x,y)a_x^{\sigma(y)}.
  $$
  Since $y(\supp f)=\supp (\bar{y}f)=\supp(f\bar{y})=(\supp f)y$, we
  get that $\supp f\subseteq Z$. Hence $C\subseteq
  K((Z;\sigma,\tau,<))$.

  Now, it is known that $G/Z$ is an orderable group (see, \textit{e.g.}~\cite[Theorem~2.24]{BMR77}).
  Since $G$ is not abelian, $G/Z$ is torsion free. If $x\in G\setminus Z$, then
  $\{\bar{x}^n : n\in\Z, n\geq 0\}$ is linearly independent over $K((Z;\sigma,\tau,<))$ inside
  $K((G;\sigma,\tau,<))$, and, therefore, over $C$. Since $\bar{x}\in D$,
  we get that $[D:C]=\infty$, as desired.
\end{proof}

\begin{lemma}\label{le:completemns}
  Let $(G,<)$ be an ordered group, let $K$ be a skew field and let
  $K[G;\sigma,\tau]$ be a crossed product of $G$ over $K$. Suppose that
  $\pi\colon G\rightarrow A$ is a surjective homomorphism of ordered groups, where
  $A$ is a nontrivial subgroup of the additive group $\R$ of the real numbers.
  Then $K((G;\sigma,\tau,<))$ is complete with respect to the valuation
  $\nu=\pi\omega\colon K((G;\sigma,\tau,<))\rightarrow A_{\infty}$.
\end{lemma}

\begin{proof}
  Fix $c\in (0,1)$ and consider the metric $d$ in
  $K((G;\sigma,\tau,<))$ induced by $\nu$,
  \textit{i.e.}~$d(f_1,f_2)=c^{\nu(f_1-f_2)}$ for $f_1,f_2\in K((G;\sigma,\tau,<))$.

  Let $(u_n)_{n\geq 1}$ be a Cauchy sequence in
  $K((G;\sigma,\tau,<))$, say $u_n=\sum_{x\in G}\bar{x}a_x^{(n)}$
  where $a_x^{(n)}\in K$. For each $k\in \mathbb{N}$ choose $n_k\in\mathbb{N}$ such that
  $d(u_p,u_q)<c^k$ whenever $p,q\geq n_k$, satisfying
  $n_0<n_1<n_2<\dots< n_k<\dotsb$. Observe that for
  $k,p,q\in\mathbb{N}$, $$d(u_p,u_q)<c^k\ \text{ if and only if } \
  a_x^{(p)}=a_x^{(q)} \text{ for all } x \text{ with } \pi(x)\leq
  k.$$

  For each $x\in G$, let $k=k(x)$ be the smallest
  $k\in\mathbb{N}$ such
  that $\pi(x)\leq k$. Define $a_x=a_x^{(n_k)}$ and let $u=\sum_{x\in G}\bar{x}a_x$.

  First we show that $u\in K((G;\sigma,\tau,<))$, that is, we show that $\supp u$ is
  well-ordered. Indeed, fix $y\in G$. Note that if $x\leq y$, then $\pi(x)\leq\pi(y)$
  because $\pi$ is a homomorphism of ordered groups. Let $k$ and $l$ be
  the smallest natural numbers such that $\pi(x)\leq k$ and $\pi(y)\leq
  l$, respectively. Then $n_k\leq n_l$, and therefore
  $a_x=a_x^{(n_k)}=a_x^{(n_l)}$ because $d(u_{n_k},u_{n_l})<c^k$.
  Hence $\{x :  x\leq y,\ a_x\neq 0\}\subseteq \supp u_{n_l}$.
  Let $S$ be a nonempty subset of $\supp u$ and  let $y\in S$. Let $l$
  be the smallest natural such that $\pi(y)\leq l$. We have that $\{x\in
  S : x\leq y\}\ne\emptyset$ and, by the foregoing argument, it is
  contained in $\supp u_{n_l}$. Thus $S$ has a least element.
  Therefore $\supp u$ is well-ordered.

  Finally, we show that $\lim u_n=u$. Given $\varepsilon>0$, there exists
  $k\in \mathbb{N}$ such that $c^k<\varepsilon$. Choose $n>n_k$. Then,
  \begin{align*}
    u-u_{n} & =   \sum_{x\in G}\bar{x}a_x-\sum_{x\in G}\bar{x}a_x^{(n)}\\
            & =   \sum_{\{x : \pi(x)\leq k\}}\bar{x}a_x
                - \sum_{\{x : \pi(x)\leq k\}}\bar{x}a_x^{(n)}
                + \sum_{\{x : \pi(x)> k\}}\bar{x}a_x
                - \sum_{\{x : \pi(x)>k\}}\bar{x}a_x^{(n)}\\
            & =   \sum_{\{x : \pi(x)> k\}}\bar{x}a_x
                - \sum_{\{x : \pi(x)>k\}}\bar{x}a_x^{(n)}.
  \end{align*}
  Hence $d(u,u_n)=c^{\nu(u-u_n)}<c^k<\varepsilon$, as desired.
\end{proof}

Now we can state and prove the main result in this section.

\begin{theorem}\label{theo:freefieldinMalcevNeumannseries}
  Let $(G,<)$ be an ordered group and let $K$ be a skew field. Consider a
  crossed product $K[G;\sigma,\tau]$ and its Malcev-Neumann series
  ring $D=K((G;\sigma,\tau,<))$. If there exists a convex jump $(N,H)$
  of $G$ such that the centre $F$ of $K(H;\sigma,\tau)$ is infinite
  and $K(H;\sigma,\tau)$ is infinite-dimensional over $F$,
  then $D$ contains a free field $\ff{C}{}{X}$, where $X$ is an
  infinite countable set and $C$ is the centre of $K(G;\sigma,\tau)$.
\end{theorem}

\begin{proof}
  The canonical projection $\pi\colon H\rightarrow H/N$ is a
  surjective homomorphism of ordered groups and $H/N$ is a nontrivial
  subgroup of $\R$. By Lemma~\ref{le:completemns}, $K((H;\sigma,\tau,<))$
  is complete with respect to the valuation $\nu=\pi\omega$. Thus, the
  completion $\widehat{K(H;\sigma,\tau)}$ of $K(H;\sigma,\tau)$ with
  respect to $\nu$ is contained in $K((H;\sigma,\tau,<))$. Hence
  Theorem~\ref{th:freefield} implies that the free field $\ff{F}{}{X}$
  is contained in $K((H;\sigma,\tau,<))\subseteq D$.

  Let $P$ denote the prime subfield of $K(H;\sigma,\tau)$. Then the free
  field $\ff{P}{}{X}$ is contained in $\ff{F}{}{X}\subseteq D$ by
  \cite[Theorem~6.4.6]{pC95}. Now, by \cite[Lemma~9]{kC03}, we get that
  $\ff{C}{}{X}$ is contained in $D$.
\end{proof}

When the crossed product is a group ring, we can find some explicit
conditions which imply the existence of a convex jump satisfying
the hypothesis of the above theorem, as the next corollary shows.

\begin{corollary}\label{cor:series}
  Let $(G,<)$ be a nonabelian ordered group and let $K$ be a skew
  field with centre $F$. Consider the group ring $K[G]$ and its
  Malcev-Neumann series ring $K((G;<))$. If
  \renewcommand{\theenumi}{\alph{enumi}}
  \begin{enumerate}
  \item $F$ is infinite, or \label{cor:cond1}
  \item the centre of $G$ is nontrivial, \label{cor:cond2}
  \end{enumerate}
  \renewcommand{\theenumi}{\arabic{enumi}}
  then $K((G;<))$ contains a free field $\ff{C}{}{X}$, where $X$
  is an infinite countable set and $C$ is the centre of $K(G)$.
\end{corollary}

\begin{proof}
  Let $x,y\in G$ be such that $x>1, y>1$ and $xy\neq yx$. In case
  \eqref{cor:cond1}, set $z=1$; in case \eqref{cor:cond2}, let $z$ be an element
  of the centre of $G$ satisfying $z>1$. Let
  $w=\max\{x,y,z\}$ and let $(N,H)$ be the convex jump associated to $w$,
  that is, $N$ is the union of all convex subgroups which do not
  contain $w$ and $H$ the intersection of all convex subgroups which
  contain $w$. In particular, $H$ is a subgroup of $G$ containing $x,y,z$.
  Let $E$ denote the centre of $K(H)$. Since $H$ is not abelian,
  Lemma~\ref{le:ffskg} implies that $[K(H):E]$ is infinite. Since
  $F[z]\subseteq E$, we get that $E$ is infinite in both cases \eqref{cor:cond1}
  and \eqref{cor:cond2}. Thus, by Theorem~\ref{theo:freefieldinMalcevNeumannseries},
  $K((G;<))$ contains a free field $\ff{C}{}{X}$ on an infinite countable set $X$.
\end{proof}

The following is an example in which the conditions of the above corollary
are satisfied.

\begin{example}
  Torsion-free nilpotent groups are known to be orderable. Moreover, a nonabelian
  torsion-free nilpotent group has a nontrivial centre. It follows from
  Corollary~\ref{cor:series} that the Malcev-Neumann series ring of such a group
  over a skew field contains a free subfield. For instance, if $G$ is a finitely
  generated torsion-free nilpotent group, then $G$ has a central series
  $$
  G=F_1\supseteq F_2\supseteq F_3\supseteq\dots\supseteq F_r\supseteq F_{r+1}=\{1\},
  $$
  in which $F_i/F_{i+1}$ is an infinite cyclic group for all $i=1,\dots,r$ (see
  \textit{e.g.}~\cite{sJ55}). Letting $f_i$ be a representative in $G$ of a generating element
  of $F_i/F_{i+1}$, it follows that the elements of $G$ can be written in a unique way
  in the form $f_1^{\alpha_1}f_2^{\alpha_2}\dots f_r^{\alpha_r}$, with $\alpha_i\in \Z$.
  Since $[F_j,F_i]\subseteq F_{j+1}$, we have
  $f_j^{\alpha_j}f_i^{\alpha_i}= f_i^{\alpha_i}f_j^{\alpha_j}
  f_{j+1}^{\gamma_{j+1}}\dots f_r^{\gamma_r}$, for some $\gamma_{j+1},\dots,\gamma_r\in\Z$,
  whenever $j>i$.
  Therefore, $G$ can be lexicographically ordered, that is, $G$ is an ordered group
  with order $<$ defined in the following way: $f_1^{\alpha_1}f_2^{\alpha_2}\dots f_r^{\alpha_r} <
  f_1^{\beta_1}f_2^{\beta_2}\dots f_r^{\beta_r}$
  if and only if there exists some $s$ for which $\alpha_s<\beta_s$, while $\alpha_i=\beta_i$ for
  all $i=1,\dots,s-1$. Now, given a skew field $K$, it is known that the group ring
  $K[G]$ is an Ore domain with field of fractions $D$, say, which must therefore
  be a subfield of $K((G;<))$. When $G$ is nonabelian, $K((G;<))$
  contains a free field $\ff{C}{}{X}$ on an infinite countable set $X$ over the centre $C$ of $D$.

  Finally, it should be remarked that $G$ can be ordered in a number of different ways. For example,
  if $G$ is the free nilpotent group of class $2$ generated by two elements,
  $$
  G = \langle x,y : [[x,y],x]=[[x,y],y]=1\rangle,
  $$
  then there is an infinite number of ways of ordering $G/\langle[x,y]\rangle\cong \Z\times \Z$
  (see \cite[Example~1.1]{NR10}). Since $[x,y]$ is central, each one of these orderings can
  be lifted to an order on $G$, and some of them do not come from a central series with
  infinite cyclic quotients. Corollary~\ref{cor:series}, though, guarantees the presence
  of free fields in $K((G;<))$ whatever order $<$ is taken.

  We shall look at this example again in Section~\ref{sec:noff}.
\end{example}

In \cite{gE06}, the concept of amenable skew fields was introduced.
We now show that Malcev-Neumann series rings provide examples of
nonamenable skew fields.

Following \cite{gE06}, let $F$ be a commutative field and let
$A$ be an $F$-algebra. We say that $A$ is
amenable over $F$ if for any finite subset $\{r_1,\dots,r_n\}\subseteq
A$ and real number $\epsilon>0$ there exists a
finite-dimensional $F$-subspace $V$ of $A$ such that
$$\frac{\dim_F(\sum_{i=1}^n r_iV)}{\dim_F V}<1+\epsilon.$$

We begin by remarking that amenability is
preserved under ground field extensions.

\begin{lemma}\label{lem:scalarextension}
  Let $F$ be a commutative field and let $K$ be a commutative field extension
  of $F$. If $A$ is an amenable $F$-algebra, then
  $K\otimes_FA$ is an amenable $K$-algebra.
\end{lemma}

\begin{proof}
  Let $\{r_1,\dotsc,r_n\}\subseteq K\otimes_FA$ and
  $\epsilon>0$. Suppose that $r_i=\sum_{l=1}^{n_i} e_{il}\otimes
  a_{il}$. Since $A$ is amenable over $F$, there exists a
  finite dimensional $F$-subspace $V$ of $A$ such that
  $$
  \frac{\dim_F \Big(\sum_{i=1}^n\sum_{l=1}^{n_i}  a_{il}V\Big)}{\dim_F V}<1+\epsilon.
  $$
  Let $W=K\otimes_F V$. Then,
  \begin{align*}
    \frac{\dim_K\Big(\sum_{i=1}^n r_i W\Big)}{\dim_K W}
      & \leq \frac{\dim_K (\sum_{i=1}^n \sum_{l=1}^{n_i}
             (1\otimes a_{il})W\Big)}{\dim_K W} \\
      & = \frac{\dim_F\Big(\sum_{i=1}^n\sum_{l=1}^{n_i} a_{il}V\Big)}%
        {\dim_F V}\\
      & < 1+\epsilon,
  \end{align*}
  as desired.
\end{proof}

\begin{proposition}\label{prop:amenabilityextension}
  Let $F$ be a commutative field, let $K$ be a commutative field
  extension of $F$ and let $X$ be a set with $|X|>1$. If $\ff{F}{}{X}$
  is amenable over $F$, then $\ff{K}{}{X}$ is amenable over $K$.
\end{proposition}

\begin{proof}
  Let $\mu\colon K\otimes_F \ff{F}{}{X} \rightarrow \ff{K}{}{X}$
  denote the multiplication map. Since $\ff{F}{}{X}$ is a simple
  algebra with centre $F$ (by \cite[Theorem~5.5.11]{pC95}),
  $K\otimes _F \ff{F}{}{X}$ is simple. Hence, $\mu$ is an
  injective homomorphism of $K$-algebras. By
  Lemma~\ref{lem:scalarextension}, $K\otimes_F\ff{F}{}{X}$ is amenable
  over $K$. By \cite[Lemma~2.1]{gE06}, $K\otimes_F\ff{F}{}{X}$ is a right
  Ore domain and its Ore field of fractions $\widetilde{K\otimes_F\ff{F}{}{X}}$
  is amenable over $K$. But $\widetilde{K\otimes_F\ff{F}{}{X}}$ is
  a subfield of $\ff{K}{}{X}$ that contains both $K$ and $X$. Thus
  $\widetilde{K\otimes_F\ff{F}{}{X}}=\ff{K}{}{X}$. Therefore $\ff{K}{}{X}$ is
  amenable over $K$.
\end{proof}

\begin{corollary}\label{cor:notamenable}
  Let $F$ be a subfield of the field of complex numbers $\C$ and
  let $X$ be a set with $|X|>1$. Then $\ff{F}{}{X}$
  is not amenable over any subfield $L$ of $F$.
\end{corollary}

\begin{proof}
  By \cite[Theorem~1(e)]{gE06}, $\ff{\C}{}{X}$ is not amenable over
  $\C$. It follows from Proposition~\ref{prop:amenabilityextension}
  that $\ff{L}{}{X}$ is not amenable over $L$ for any subfield $L$
  of $\C$. If $L\subseteq F$, then $\ff{L}{}{X}\subseteq \ff{F}{}{X}$,
  by \cite[Theorem~6.4.6]{pC95}. It follows from
  \cite[Proposition~3.1]{gE06}, that $\ff{F}{}{X}$ is not amenable over $L$.
\end{proof}

We are now ready to show that, under mild conditions, Malcev-Neumann
series rings are nonamenable skew fields.

\begin{corollary}\label{coro:MalcevNeumannnotamenable}
  Let $F$ be a subfield of the field of complex numbers $\C$. Let $K$ be a skew
  field with centre $Z$, let $(G,<)$ be an ordered group and let $K[G;\sigma,\tau]$
  be a crossed product which is an $F$-algebra.  If either
  \renewcommand{\theenumi}{\alph{enumi}}
  \begin{enumerate}
    \item $G$ is nonabelian, or \label{cor:am1}
    \item $G$ is abelian, $[K:Z]=\infty$ and $K[G;\sigma,\tau]=K[G]$, \label{cor:am2}
  \end{enumerate}
  \renewcommand{\theenumi}{\arabic{enumi}}
  then the Malcev-Neumann series ring $K((G;\sigma,\tau,<))$ is not an
  amenable skew field over any subfield $L$ of $F$.
\end{corollary}

\begin{proof}
  Because $G$ is ordered, the only units in $K[G;\sigma,\tau]$
  are the trivial ones. Hence, the fact that $K[G;\sigma,\tau]$
  is an $F$-algebra implies that $F\subseteq Z$.
  In case \eqref{cor:am1}, let $(N,H)$ be a convex jump of $(G,<)$ with $H$ nonabelian.
  Let $E$ denote the centre of $K(H;\sigma,\tau)$. Then $F\subseteq E$ and,
  thus, $E$ is infinite. By Lemma~\ref{le:ffskg}, $[K(H;\sigma,\tau):E]$ is
  infinite. Hence, Theorem~\ref{theo:freefieldinMalcevNeumannseries}
  yields that $K((G;\sigma,\tau,<))$ contains a free field
  $\ff{C}{}{X}$, where $X$ is an infinite countable set and $C$ is the centre
  of $K(G;\sigma,\tau)$. Since
  $F\subseteq C$, \cite[Theorem~6.4.6]{pC95} implies that
  $\ff{F}{}{X}\subseteq\ff{C}{}{X}$.
  In case \eqref{cor:am2}, by Corollary~\ref{cor:MNbig}, $K((G,<))$ contains
  a free field $\ff{Z}{}{X}$, where $X$ is an infinite countable set. Again,
  by \cite[Theorem~6.4.6]{pC95}, $\ff{Z}{}{X}$ contains $\ff{F}{}{X}$.

  In both cases, $K((G;\sigma,\tau,<))$ contains a free field
  $\ff{F}{}{X}$, for infinite countable set $X$. By Corollary~\ref{cor:notamenable},
  $\ff{F}{}{X}$ is not amenable
  over any subfield $L$ of $F$.
  It follows from \cite[Proposition~3.1]{gE06} that $K((G;\sigma,\tau,<))$ is not
  an amenable skew field over $L$.
\end{proof}

As a consequence, we note
that if $K$ is a skew field of zero characteristic and $(G,<)$ is a nonabelian
ordered group, then $K((G;\sigma,\tau,<))$ is a nonamenable skew field over $\mathbb{Q}$.

\section{Skew Laurent series}\label{sec:ls}

In this section we show that fields of fractions of skew polynomial
rings over skew fields are
endowed with natural valuations whose completions, skew Laurent series rings,
contain free fields. In particular, it will be shown that the Weyl field can be
embedded in a complete skew field containing a free field.

In the absence of derivations, the skew fields considered in this
section can be treated by the methods of the previous section. We,
therefore, concentrate in the case where a derivation is present.

Let $A$ be a ring, let $\sigma$ be an endomorphism of $A$ and let
$\delta$ be a (right) $\sigma$-derivation of $A$, that is,
$\delta\colon A\rightarrow A$ is an additive map satisfying
$$
\delta(ab) = \delta(a)\sigma(b)+a\delta(b),
$$
for all $a,b\in A$. The \emph{skew polynomial ring in $x$ over $A$},
hereafter denoted by $A[x;\sigma,\delta]$, is the free right
$A$-module on the nonnegative powers of $x$ with multiplication
induced by
$$
ax = x\sigma(a)+\delta(a),
$$
for all $a\in A$. If $A$ is a skew field then $A[x; \sigma, \delta]$
is a principal right ideal domain and its field of fractions will be
denoted by $A(x; \sigma, \delta)$. When $A$ is a right Ore domain
with field of fractions $Q(A)$ and $\sigma$ is injective, $\sigma$
and $\delta$ can be extended to $Q(A)$ and we can consider the skew
polynomial ring $Q(A)[x; \sigma, \delta]$. In this case we have
$$
A[x; \sigma,\delta] \subseteq Q(A)[x; \sigma, \delta] \subseteq
Q(A)(x; \sigma,\delta)
$$
and so $A[x; \sigma, \delta]$ is a right Ore domain with field of
fractions $Q(A)(x; \sigma, \delta)$.

For an arbitrary endomorphism $\sigma$ of the skew field $K$ a
further construction will be considered. The ring of \emph{skew
power series} $K[[ x;\sigma]]$ is defined to be the set of power
series on $x$ of the form $\sum_{i=0}^{\infty} x^ia_i$, with $a_i\in
K$, where addition is done coefficientwise and multiplication is
given by the rule $ax=x\sigma(a)$, for $a\in K$. When $\sigma$ is an
automorphism, we can embed $K[[x;\sigma]]$ into the ring of
\emph{skew Laurent series} $K((x;\sigma))$ whose elements are series
$\sum_{-r}^{\infty}x^ia_i$, with $r$ a nonnegative integer. In
$K((x;\sigma))$, multiplication satisfies $ax^n=x^n\sigma^n(a)$, for
every integer $n$ and $a\in K$. It is well known that
$K((x;\sigma))$ is a skew field containing $K(x;\sigma)$.

As mentioned in \cite[Section~2.3]{pC95}, when $\delta\neq 0$ we
must look at a more general construction. Given an automorphism
$\sigma$ of a skew field $K$ and a $\sigma$-derivation $\delta$ on
$K$, denote by $R$ the ring of all power series
$\sum_{i=0}^{\infty}a_iy^i$ on $y$ with multiplication induced by
$ya = \sum_{i=0}^{\infty}\sigma\delta^i(a)y^{i+1}$, for $a\in K$. In
\cite[Theorem 2.3.1]{pC95} it is shown that $R$ is a domain and
$S=\{1, y, y^2,\dots\}$ is a left Ore set in $R$ whose localization
$S^{-1}R$ is a skew field, consisting of all skew Laurent series
$\sum_{i=r}^{\infty} a_iy^i$, with $r\in\mathbb{Z}$. Because in
$S^{-1}R$ we have $ya=\sigma(a)y+y\delta(a)y$, for all $a\in K$, it
follows that $ay^{-1}=y^{-1}\sigma(a)+\delta(a)$. Hence, the subring
generated by $y^{-1}$ in $S^{-1}R$ is an ordinary skew polynomial
ring on $y^{-1}$. Because of that we shall adopt the notation
$R=K[[x^{-1}; \sigma, \delta]]$ and $S^{-1}R=K((x^{-1}; \sigma,
\delta))$. So we have an embedding $K[x;\sigma,\delta]\subseteq K((
x^{-1}; \sigma, \delta))$, sending $x$ to $y^{-1}$ and, hence,
$K(x;\sigma,\delta)\subseteq K((x^{-1};\sigma,\delta))$. It is also
clear that $K((x^{-1};\sigma,\delta))=K[[x^{-1}; \sigma,
\delta]]+K[x;\sigma, \delta]$.

The field $K(x;\sigma,\delta)$ has a discrete valuation ``at infinity'' $\nu$
which extends the valuation $-\deg$ on $K[x; \sigma, \delta]$, where
$\deg$ denotes the usual degree function on $K[x;\sigma,\delta]$,
that is, $\deg(\sum x^ia_i) = \max\{i : a_i\ne 0\}$.

\begin{lemma} \label{le:completepol}
  Let $K$ be a skew field with an automorphism $\sigma$ and a $\sigma$-derivation
  $\delta$. Then the map $\omega\colon K((x^{-1};\sigma,\delta)\rightarrow\Zi$
  defined by
  $$
  \omega(f)=\sup\{n : f\in K[[x^{-1}; \sigma, \delta]] y^n\},
  $$
  is a valuation on $K((x^{-1};\sigma,\delta))$ which extends the
  valuation $\nu$ on $K(x;\sigma,\delta)$. Moreover, $K((x^{-1};\sigma,\delta))$ is
  the completion of $K(x;\sigma,\delta)$.
\end{lemma}

\begin{proof}
  To show that $\omega$ is a valuation, the only nonobvious fact to be proved is
  that if $f$ and $g$ are nonzero elements of $K((x^{-1};\sigma,\delta))$, then
  $\omega(fg) = \omega(f)+\omega(g)$. This follows from the fact that, since $\sigma$ is
  invertible, it is possible to write, for every integer $n$ and $a\in K$,
  $y^na=\sigma^n(a)y^n+hy^{n+1}$, for some $h\in K[[x^{-1}; \sigma, \delta]]$.

  Clearly, $\omega_{\mid K(x; \sigma, \delta)}=\nu$. That $K(x;\sigma,\delta)$ is dense
  in $K((x^{-1};\sigma,\delta))$ follows from the fact that every Laurent series
  $\sum_{i=n}^{\infty}a_iy^i$ in $K((x^{-1};\sigma,\delta))$ is the limit of its partial
  sums, which are elements in $K(x;\sigma,\delta)$. It remains to prove that
  $K((x^{-1}; \sigma, \delta))$ is complete. Let $(u_n)$ be a Cauchy sequence in
  $K((x^{-1}; \sigma, \delta))$, say
  $$
  u_n=\sum_{i\geq m_n}a_i^{(n)}y^i,
  $$
  with $a_i^{(n)}\in K$ and $m_n\in\Z$. We proceed as in the proof
  of Lemma~\ref{le:completemns}. For each $k\in \Z$ choose nonnegative
  integers $n_k$ satisfying $n_0 < n_1 < n_2 < \dots$
  such that $\omega(u_p-u_q)>k$, for all $p,q\geq n_k$. Consider the following
  element of $K((x^{-1};\sigma,\delta))$,
  $$
  u=\sum_{i<0}a_i^{(n_0)}y^i+\sum_{i\geq 0}a_i^{(n_i)}y^i.
  $$
  For all $k$,
  $\omega(u_{n_k}-u)\geq k+1$ and, hence, the subsequence $(u_{n_k})$
  of $(u_n)$ converges to $u$. Since $(u_n)$ is a Cauchy sequence,
  $(u_n)$ converges to $u$. It follows that $K((x^{-1};\sigma,\delta))$
  is complete.
\end{proof}

In the case that $\delta=0$, $K(x; \sigma)$ has an $x$-adic discrete
valuation $\eta$, defined by $\eta(fg^{-1})=o(f)-o(g)$, for all
$f,g\in K[x;\sigma]$, $g\ne 0$, where $o(\sum x^ia_i)=\min\{i :
a_i\neq 0\}$ is the order function on $K[x;\sigma]$, which is a
valuation. Denote by $\zeta$ the integer valued function on
$K((x;\sigma))$ defined by $\zeta(h)=\sup\{n : h\in x^n K[[
x;\sigma]]\}$, for all $h\in K((x;\sigma))$, $h\ne 0$. It is easy to
see that $\zeta$ is a discrete valuation on $K((x;\sigma))$ which coincides
with $\eta$ on $K(x;\sigma)$. With a proof similar to the one of
Lemma~\ref{le:completepol} it can be shown that $K(( x; \sigma))$ is
the completion of $K(x; \sigma)$ with respect to $\eta$.

Our next aim is to guarantee the existence of free fields in
completions of skew rational function fields. For that, we recall
the following definition. An automorphism $\sigma$ in a ring is said
to have \emph{inner order} $r$ if $\sigma^r$ is the least positive
power of $\sigma$ which is inner. If $\sigma^r$ is not inner for
any $r>0$, then $\sigma$ is said to have infinite inner order.

\begin{theorem}\label{th:ffpol}
  Let $K$ be a skew field with centre $F$, let $\sigma$ be an automorphism
  of $K$ and let $\delta$ be a $\sigma$-derivation of $K$. Suppose that
  $K[x; \sigma, \delta]$ is simple or that $\sigma$ has infinite inner order.
  If $\delta\ne 0$, let $D=K((x^{-1};\sigma,\delta))$, otherwise, let
  $D=K((x;\sigma))$. If the field $F_0=\{c\in F : \sigma(c)=c
  \text{ and } \delta(c)=0\}$ is infinite, then $D$ contains a free field
  $\ff{F_0}{}{X}$ on an infinite countable set $X$.
\end{theorem}

\begin{proof}
  Since $K[x; \sigma, \delta]$ is simple or $\sigma$ has infinite
  inner order, it follows from \cite[Theorem~2.2.10]{pC95} that the centre of $K(x; \sigma,
  \delta)$ coincides with $F_0$. The set $\{x^i : i\in\Z\}$ is linearly independent
  over $F_0$, because it is over $K$; hence $K(x;\sigma,\delta)$ is
  infinite dimensional over its infinite centre $F_0$. By Theorem \ref{th:freefield},
  $D$ contains a free field on an infinite countable set over $F_0$, for, as we have seen
  above, it is the completion of $K(x;\sigma,\delta)$ with respect
  to appropriate discrete valuations.
\end{proof}

We remark that the case $\delta=0$ in the above theorem had already
been considered in \cite[Theorem~4]{kC03} for a commutative field
$K$ of coefficients.

An example of a skew polynomial ring with a nonzero derivation is
the first Weyl algebra over a skew field $K$,
$$
A_1(K)=\fr{K}{}{x_1,x_2 :
x_1x_2-x_2x_1=1}=K[x_1]\left[x_2;I,\frac{d}{dx_1}\right],
$$
where $I$ above stands for the identity automorphism of $K[x_1]$. It
is well known that $A_1(K)$ is an Ore domain with field of fractions
$Q_1(A)=K(x_1)(x_2;I,\frac{d}{dx_1})$, called the Weyl field.

\begin{corollary}
  Let $K$ be a skew field with centre $F$. If $\charac K=0$ then
  the field of skew Laurent series $K(x_1)((x_2^{-1};I,\frac{d}{dx_1}))$ contains a free
  field $\ff{F}{}{X}$, where $X$ is an infinite countable set.
\end{corollary}

\begin{proof}
  Let $B=K(x_1)[x_2;I,\frac{d}{dx}]$. Then $A_1(K)\subseteq B\subseteq
  K(x_1)(x_2; I,\frac{d}{dx_1})$ and $Q(B)=Q_1(A)=K(x_1)(x_2;I,\frac{d}{dx_1})$.
  Clearly, $\frac{d}{dx_1}$ is not an inner
  derivation. It follows from \cite[Corollary 3.16]{tL01} that $B$ is
  a simple ring. Since $\charac K=0$, we have $Z(K(x_1))_0=\{f\in
  Z(K(x_1)) : \frac{d}{dx_1}(f)=0\}=F(x_1)_0=F$, an infinite field.
  By Theorem \ref{th:ffpol}, the field of
  skew Laurent series $K(x_1)((x_2^{-1}; I,\frac{d}{dx_1}))$
  contains a free field on an infinite countable set over $F$.
\end{proof}

A final word will be given on a related problem, that of
division rings generated by Lie algebras. Given a nonabelian Lie
algebra $L$ over a commutative field $k$, its universal
enveloping algebra $U(L)$ has a natural filtration given
by the powers of $L$ inside $U(L)$, and this filtration defines a
valuation $\nu$ on $U(L)$. (See \cite[Section~2.5]{pC95} or \cite{aL95}.)
By \cite[Theorem~2]{aL95}, $U(L)$ has a field of
fractions $D(L)$ with a valuation extending $\nu$.
Theorem~\ref{th:freefield}, then, implies that
if $\charac k =0$ the
completion $\overline{D(L)}$ of $D(L)$ contains a free
field $\ff{C}{}{X}$, for an infinite countable set $X$,
where $C$ stands for the centre of $D(L)$. Note that
Theorem~\ref{th:freefield} can be applied in this
setting because, by \cite{aL99}, $D(L)$ contains
a free subalgebra.

\section{Some skew fields that do not contain free fields}\label{sec:noff}

In this section we shall present a theorem giving a necessary
condition for a division ring to contain a free field. This will
then be used in order to produce examples of division rings which,
although containing free algebras, do not contain free fields.

The results in this section are due to A.~Lichtman, but have not
been published before. The authors are indebted to him for having
kindly permitted this material to be included in the paper.

Throughout this section, $k$ will denote a commutative field.

Let $A$ be a $k$-algebra. Denote by $A^{\op}$ the opposite algebra
of $A$. Regard the $(A,A)$-bimodule $A$ as a left module over the
enveloping algebra $A\otimes_k A^{\op}$ in the usual way, that is,
via $a\otimes b\cdot x = axb$, for all $a,b,x\in A$. We shall make
use of the following result of Sweedler \cite{mS80}.

\begin{lemma}\label{le:Lichtman}
  Let $D$ be a skew field which is a $k$-algebra. If $D\otimes_k D^{\op}$
  is a left noetherian algebra, then $D$ does not contain
  an infinite strictly ascending chain of subfields $D_1\subsetneqq D_2 \subsetneqq D_3
  \subsetneqq \dots$ which are $k$-subalgebras. \qed
\end{lemma}

For a direct proof we refer to \cite{vB08}.

As a consequence we obtain the following necessary condition for a
skew field to contain a free field.

\begin{theorem}\label{th:nofreefield}
  Let $D$ be a skew field and let $k$ be a central subfield of $D$.
  If $D$ contains a free field $\ff{k}{}{x,y}$ then $D\otimes_k D^{\op}$
  cannot be a left noetherian ring.
\end{theorem}

\begin{proof}
  If $D$ contains the free field $\ff{k}{}{x,y}$ then, by \cite[Corollary
  5.5.9]{pC95}, $D$ contains a free
  field with an infinite countable number of generators, say,
  $\ff{k}{}{x_1,x_2,\dots}\subseteq D$. For every $n\geq 1$, letting
  $D_n=\ff{k}{}{x_1,\dots,x_n}$, we obtain an infinite strictly ascending
  chain of subfields in $D$, $D_1\subsetneqq D_2\subsetneqq\dots$.
  It follows from Lemma~\ref{le:Lichtman} that $D\otimes_k D^{\op}$ is not a left
  noetherian ring.
\end{proof}

Now consider a nonabelian torsion-free polycyclic-by-finite group
$G$. We recall that the group ring $K[G]$ is a noetherian domain for
every skew field $K$ (cf.~\cite[Proposition~1.6 and
Corollary~37.11]{dP89}). By Theorem~\ref{th:nofreefield} we have the
following result.

\begin{corollary}\label{cor:group}
  The field of fractions of the group algebra of a nonabelian
  torsion-free polycyclic-by-finite group over a commutative field
  does not contain a noncommutative free field (over any subfield).
\end{corollary}

\begin{proof}
  Let $G$ be a nonabelian torsion-free polycyclic-by-finite group. Let
  $F$ denote the left classic field of fractions of $k[G]$ and consider the
  $k$-algebra $R=k[G]\otimes_k F^{\op}\cong F^{\op}[G]$. Since $k[G]$ is a left
  Ore domain, $T=(k[G]\setminus\{0\})\otimes_k 1$ is a left denominator
  subset of $R$. Since $R$ is left noetherian, 
  $T^{-1}R\cong F\otimes_k F^{\op}$ is also left noetherian. By Theorem~\ref{th:nofreefield},
  $F$ does not contain a free field over $k$. Therefore, $F$ does not contain a
  free field over any central subfield, by \cite[Lemma~9]{kC03}. Consequently,
  $F$ does not contain a free field over any subfield, by \cite[Corollary~p.~114]{pC77}
  and \cite[Theorem~5.8.12]{pC95}.
\end{proof}

Let $H$ be a nonabelian torsion-free finitely generated nilpotent
group and consider the group algebra $k[H]$. Since $H$ is
polycyclic-by-finite then the group algebra $k[H]$ is a noetherian
domain. In \cite{lM84}, Makar-Limanov showed that the field of
fractions of $k[H]$ contains noncommutative free algebras over $k$
and, consequently, it contains fields of fractions of free
algebras. However, none of these fields of fractions is a free field, by
Corollary~\ref{cor:group}.

Proceeding in a similar fashion, it is possible to show
that the field of fractions $Q$ of the first Weyl algebra over a
commutative field $k$ does not contain a noncommutative free field
(over any subfield). We remark that, again, Makar-Limanov proved in
\cite {lM83} that if $k$ has zero characteristic, $Q$ contains
noncommutative free algebras over $k$. By what we have just seen, the
subfields of $Q$ generated by these free algebras are not free fields.

We remark that the methods of G.~Elek \cite{gE06} provide an alternative
proof for Corollary~\ref{cor:group} in the case of algebras over
the field of complex numbers.

\section{Free fields of uncountable rank in Malcev-Neumann series rings}

This final section is devoted to the task of showing that Chiba's ideas
in \cite{kC03} can be explored further with the aim of showing that
certain Malcev-Neumann series rings contain free fields of uncountable
rank.

We start by recalling \cite[Lemma~1]{kC03}, which will be used in
the sequel.

\begin{lemma}\label{lem:specializationChiba}
  Let $D$ be a skew field with infinite centre and let $X$ be a set. Let $K$
  be a subfield of $D$ which is its own bicentralizer and whose
  centralizer $K'$ is such that the left $K$-space $KcK'$ is
  infinite-dimensional over $K$ for all $c\in D^\times$. Let $L$ be a
  noncentral subnormal subgroup of the multiplicative group
  $K'^\times$, equivalently, $L$ is a subnormal subgroup of $K'^\times$
  such that $L\nsubseteq K$. Then any full matrix over $D_K\langle
  X\rangle$ is invertible for some choice of values of $X$ in $L$. \qed
\end{lemma}

The following result is a slightly more general version of
\cite[Lemma~2]{kC03}. Note that we use multiplicative notation for
the  value group of $\nu$.

\begin{lemma}\label{lem:notcontainedinthecentre}
  Let $D$ be a skew field with a valuation $\nu$ on an ordered group
  $(G,<)$, and let $K$ be a subfield of $D$ which is its own
  bicentralizer and whose centralizer $K'$ is such that the left
  $K$-space $KcK'$ is infinite-dimensional over $K$, for all $c\in
  D^\times$. Let $L=\{x\in K' : \nu(x)=1\}$. Then $L$ is a normal
  subgroup of the multiplicative group of $K'^\times$ and $L\nsubseteq
  K$.
\end{lemma}

\begin{proof}
  For $c=1$, we get that $KK'$ is infinite dimensional over $K$. Note
  also that the centre of $K'$ is contained in $K$ since it is $K\cap
  K'$, thus $K'$ is infinite dimensional over its centre.

  Suppose that $L\subseteq K$. Let $x\in K'\setminus (K\cap K')$ with
  $\nu(x)>1$. Such an $x$ exists because $K'$ is infinite-dimensional
  over $K\cap K'$ and $L\subseteq K$. Observe that there exists
  $y\in K'$ with $xy\neq yx$. Now $\nu(1-x)=\min\{\nu(1),\nu(x)\}=1$
  but $y(1-x)=y-yx\neq y-xy=(1-x)y$. Thus $1-x\in L$, but $1-x\notin K$,
  a contradiction.
\end{proof}

Let $(G,<)$ be an ordered group and let $\Gamma$ be the chain of its convex
subgroups. Order $\Gamma$ with respect to inclusion, that is, given
$H_1,H_2\in \Gamma$, set $H_1\prec H_2$ if and only if
$H_1\subsetneq H_2$. The order type of $(\Gamma,\prec)$ is an
invariant of $(G,<)$. Consider the subset $\Gamma_0$ of
\emph{principal convex subgroups}, that is, given $H\in \Gamma$, then
$H\in\Gamma_0$ if there exists an element $t\in H$ such that $H=\{x\in
G : t^{-n}\leq x\leq t^{n}, \text{ for some natural number $n$}\}$. A
convex subgroup $H$ is principal if and only if it is the greater
member of a convex jump $(N,H)$ (see, for example, \cite{lF63}).

Let $K$ be a skew field, let $(G,<)$ be an ordered group, let
$K[G;\sigma,\tau]$ be a crossed product group ring and let
$K((G;\sigma,\tau,<))$ be its Malcev-Neumann series ring. Now let $I$ be a set and
suppose that there exists a map $I\rightarrow K((G;\sigma,\tau,<)), i\mapsto
f_i=\sum_{x\in G}\bar{x} a_{ix},$ such that the following two
conditions hold:
\begin{enumerate}
  \item $\bigcup_{i\in I}\supp(f_i)$ is well-ordered, and
  \item for each $x\in G$ the set $\{i\in I : x\in\supp(f_i)\}$ is finite.
\end{enumerate}
Then, following \cite{DL82}, we say that $\sum\limits_{i\in I}f_i$ is
\emph{defined in $K((G;\sigma,\tau,<))$}. When this is the case, $\sum\limits_{i\in I}f_i$
will be used to denote $\sum_{x\in G}\bar{x}\left(\sum_{\{i : x\in
\supp(f_i)\}} a_{ix}\right)$. Note that $\sum\limits_{i\in I}f_i$ is
then an element of $K((G;\sigma,\tau,<)).$

\begin{lemma}\label{lem:welldefinedsum}
  Let $(G,<)$ be an ordered group, let $K$ be a skew field and let
  $K[G;\sigma,\tau]$ be a crossed product. Let $\alpha$ be an ordinal.
  Suppose that the chain $(\Gamma_0,\prec)$ of principal convex
  subgroups of $(G,<)$ contains a subset
  $\{H_\beta\}_{\beta\leq\alpha}$ which is order-isomorphic to $\{\beta :
  \beta\leq\alpha\}$. For all $\beta\leq\alpha$, fix $1<t_\beta$ such
  that $H_\beta=\{x : t_\beta^{-n}\leq x\leq t^n_\beta,\text{ for
  some natural number $n$}\}$. If $\{f_\beta\}_{\beta\leq\alpha}$ is a
  sequence in $K((G;\sigma,\tau,<))$ with $\supp f_\beta\subseteq
  \{x\in H_\beta : x\geq 1\}$ for all $\beta\leq \alpha$, then
  $\sum_{\beta\leq\alpha}\bar{t}_\beta f_\beta$ is defined in
  $K((G;\sigma,\tau,<)).$
\end{lemma}

\begin{proof}
  First note that if $\beta_1<\beta_2\leq\alpha$, then
  $t_{\beta_2}^{-1}<t_{\beta_1}^{-n}\leq
  t_{\beta_1}^{n}<t_{\beta_2}$, for all natural numbers $n$, for
  otherwise $t_{\beta_2}\in H_{\beta_1}$ and thus
  $H_{\beta_1}=H_{\beta_2}$, a contradiction. Hence, since $\supp
  f_\beta\subseteq \{x\in H_\beta : x\geq 1\}$ and $H_\beta=\{x\in
  G : t_\beta^{-n}\leq x\leq t_\beta^n, \text{ for some natural number $n$}\}$,
  for each $\beta\leq \alpha$, we get that $x_1<x_2$ for each
  $\beta_1<\beta_2$ and $x_1\in\supp \bar{t}_{\beta_1}f_{\beta_1}$,
  $x_2\in \supp \bar{t}_{\beta_2}f_{\beta_2}$. Therefore, for each
  $x\in G$, the set $\{\beta\leq \alpha : x\in\supp
  \bar{t}_{\beta}f_{\beta}\}$ is finite.

  Now let $S$ be a nonempty subset of $\cup_{\beta\leq\alpha}\supp
  \bar{t}_\beta f_{\beta}$. Let $\beta_0=\min \{\beta : S\cap\supp
  \bar{t}_\beta f_\beta\neq\emptyset\}$. Then the least element of the
  well-ordered set $S\cap \supp \bar{t}_{\beta_0} f_{\beta_0}$ is the
  least element of $S$. Hence $\sum_{\beta\leq \alpha}\bar{t}_\beta
  f_\beta$ is defined in $K((G;\sigma,\tau,<))$.
\end{proof}

The proof of the next result is similar to that of
\cite[Theorem~1]{kC03}. We recall the following fact that will be
used in the proof. Let $D$ be a skew field with centre $C$. Let $X$
be a set and consider $\fr{D}{C}{X}$. Let $a,t\in D$, with $t\ne 0$. Then there exists
a unique isomorphism of $D$-rings $\fr{D}{C}{X}\rightarrow \fr{D}{C}{X}$ with
$x\mapsto a+tx$, for all $x\in X$. Hence, if $A(x)$ is a full matrix
over $\fr{D}{C}{X}$, then $A(a+tx)$ is also a full matrix over $\fr{D}{C}{X}$.

\begin{theorem}\label{theo:freefielduncountablerank}
  Let $(G,<)$ be an ordered group, let $K$ be a skew field and let
  $K[G;\sigma,\tau]$ be a crossed product. Let $\zeta$ be an infinite
  cardinal and let $X=\{x_\gamma\}_{\gamma<\zeta}$ be a set of cardinality
  $\zeta$. Suppose that the following two conditions hold true.
  \renewcommand{\theenumi}{\alph{enumi}}
  \begin{enumerate}
  \item The chain $(\Gamma_0,\prec)$ of principal convex subgroups of $(G,<)$
    contains a subset $\{H_\alpha\}_{\alpha<\zeta}$ which is order-isomorphic to
    $\zeta$. \label{conda}
  \item For all $\alpha<\zeta$, the centre $C_\alpha$ of
    $E_\alpha=K((H_\alpha;\sigma,\tau,<))$ is infinite and
    $[E_\alpha:C_\alpha]=\infty$. \label{condb}
  \end{enumerate}
  \renewcommand{\theenumi}{\arabic{enumi}}
  Then $K((G;\sigma,\tau,<))$ contains a free field $\ff{C}{}{X}$,
  where $C$ denotes the centre of $K((G;\sigma,\tau,<))$.
\end{theorem}

\begin{proof}
  Let $P$ be the prime subfield of $C$. Let $\Sigma$ be
  the set of all full matrices over $P\langle X\rangle$. Then $\Sigma$
  is of cardinality $\zeta$. Write
  $\Sigma=\{A_\alpha(x_\gamma)\}_{\alpha<\zeta}$. For each
  $\alpha<\zeta$, fix $1<t_\alpha\in H_\alpha$  such that
  $H_\alpha=\{x\in G : t_\alpha^{-n}\leq x\leq t_\alpha^n, \text{
  for some natural number $n$}\}$.

  By \cite[Lemma~9]{kC03}, it is enough to prove that $P(\langle
  X\rangle)$ embeds into $E=K((G;\sigma,\tau,<))$. We will show that
  there exists an honest embedding $\fr P{}X\hookrightarrow E$, that
  is, there exists $\{d_\gamma\}_{\gamma<\zeta}\subseteq E$ such that
  the matrix $A_\alpha(d_\gamma)$ is invertible for all
  $\alpha<\zeta$. The first part of the proof is devoted to the
  construction of these $d_\gamma,$ $\gamma<\zeta$.

  Let $\omega\colon E\rightarrow G_{\infty}$ denote
  the natural valuation given by $\omega(f) = \min\supp f$, for all
  nonzero $f\in E$.
  We define sequences $\{h_{\gamma\beta} : \gamma,\beta <\zeta\}$
  satisfying the following conditions.
  \renewcommand{\theenumi}{\roman{enumi}}
  \begin{enumerate}
  \item $h_{\gamma\alpha}\in V_{\alpha}'=\{f\in E_\alpha : \omega(f)\geq
    1\}$, for all $\alpha<\zeta$. \label{c1}
  \item For each $\beta$, $h_{\gamma\beta}=0$ for almost all but a
    finite number of $\gamma$. \label{c2}
  \item $\sum_{\beta\leq \alpha} \bar{t}_\beta h_{\gamma\beta}$ is defined in $E_\alpha$.
    \label{c3}
  \item $A_\alpha(\sum_{\beta\leq\alpha}\bar{t}_\beta h_{\gamma\beta})$ is invertible
    over $E_\alpha$, for all $\alpha<\zeta$. \label{c4}
  \end{enumerate}
  \renewcommand{\theenumi}{\arabic{enumi}}

  First, observe that \eqref{c3} always holds by Lemma~\ref{lem:welldefinedsum}.

  Set $L_0=\{x\in E_0 : \omega(x)=1\}$. Then $L_0\nsubseteq C_0$, by
  Lemma~\ref{lem:notcontainedinthecentre}. Consider the finite set of
  $x$'s which appear in the matrix $A_0(x_\gamma)$, \textit{e.g.}~$x_0,x_1,\dots,x_r$.
  Since $A_0(\bar{t}_0x_\gamma )$ is a full matrix
  over $\fr{E_0}{C_0}{X}$, there are elements $h_{\gamma
  0}\in L_0$, $\gamma=1,\dots,r$, such that
  $A_\gamma(\bar{t}_0h_{\gamma0})$ is an invertible matrix over $E_0$, by
  Lemma~\ref{lem:specializationChiba}. Set
  $h_{\gamma 0}=0$ for $\gamma\notin\{1,\dotsc,n\}$. Let $B_0$ be the
  inverse of $A_0(\bar{t}_0h_{\gamma0})$. There exists a natural number $b_0>0$
  such that $B_0=\bar{t}_0^{-b_0}B_0^0 $, where all entries of $B_0^0$ are
  elements of $V_{01}=\{f\in E_0 : \omega(f)\geq 1\}$.

  Let $0<\alpha<\zeta$. Suppose that we have defined $h_{\gamma\beta}$
  for $\gamma<\zeta$, $\beta<\alpha$ satisfying \eqref{c1}--\eqref{c4}.
  Set $L_\alpha=\{f\in E_\alpha : \omega(f)=1\}$. Then
  $L_\alpha\nsubseteq C_\alpha$, by
  Lemma~\ref{lem:notcontainedinthecentre}. Consider the finite set
  $\{x_{\gamma_1},\dots,x_{\gamma_s}\}\subseteq X$ which appear in the
  matrix $A_\alpha(x_\gamma)$. Since
  $A_\alpha\big(\sum_{\beta<\alpha}\bar{t}_\beta h_{\gamma \beta} +\bar{t}_\alpha
  x_\alpha \big)$ is a full matrix over $\fr{E_\alpha}{C_\alpha}{X}$,
  there are $h_{\gamma\alpha}\in L_\alpha$,
  $\gamma=\gamma_1,\dots,\gamma_s$, such that
  $A_\alpha\big(\sum_{\beta\leq \alpha}\bar{t}_\beta h_{\gamma\beta}\big)$
  is an invertible matrix over $E_\alpha$, by
  Lemma~\ref{lem:specializationChiba}. Set $h_{\gamma\alpha}=0$ for
  $\gamma\notin\{\gamma_1,\dotsc,\gamma_s\}$. Let $B_\alpha$ be the
  inverse matrix of $A_\alpha\big(\sum_{\beta\leq\alpha}\bar{t}_\beta
  h_{\gamma\beta}\big)$. There exists a natural number $b_\alpha>0$ such that
  $B_\alpha=\bar{t}_\alpha^{-b_\alpha}B_\alpha^0$, where all the entries of
  $B_\alpha^0$ are elements $V_{\alpha 1}=\{f\in E_\alpha : \omega(f)\geq 1\}$.

  Now it is easy to see that
  $d_\gamma=\sum_{\alpha<\zeta}\bar{t}_\alpha h_{\gamma\alpha}$
  is defined in  $E$, for
  each $\gamma<\zeta$.

  We now show that $A_\alpha(d_\gamma)=A_\alpha
  \big(\sum_{\alpha<\zeta}\bar{t}_\alpha h_{\gamma\alpha}\big)$ is an
  invertible matrix over $E$, for all
  $\alpha<\zeta$. For each $\alpha<\zeta$, we have
  $$
  B_\alpha^0 \cdot A_\alpha
    \Big(\sum_{\beta\leq\alpha}\bar{t}_\beta h_{\gamma\beta}\Big)=
    \begin{pmatrix}
      \bar{t}_\alpha^{b_\alpha} & & 0 \\
      & \ddots & \\
      0 & & \bar{t}_\alpha^{b_\alpha}
    \end{pmatrix}
  $$
  and
  $$
  \sum_{\beta<\zeta}\bar{t}_\beta h_{\gamma\beta}\equiv
  \sum_{\beta\leq\alpha}\bar{t}_\beta h_{\gamma\beta}
  \mod (V_{t_{\alpha+1}}),
  $$
  where $V_{t_{\alpha+1}}=\{f\in E : \omega(f)\geq
  t_{\alpha+1}\}$. Hence,
  $$
  B_\alpha^0\cdot A_\alpha
  \Big(\sum_{\beta<\zeta}\bar{t}_\beta h_{\gamma\beta}\Big)=
  \begin{pmatrix}
    \bar{t}_\alpha^{b_\alpha} & & 0\\
    & \ddots & \\
    0 & & \bar{t}_\alpha^{b_\alpha}
  \end{pmatrix}+ M',
  $$
  where the entries of $M'$ are in $V_{t_{\alpha+1}}$. Since
  $\{\bar{t}_\alpha^{b_\alpha}(1-M)\}^{-1}=(1+M+M^2+\dotsb)\bar{t}^{-b_\alpha}$,
  where $M$ is a matrix over $V_1=\{f\in E : \omega(f)>1\}$, it follows that
  $A_\alpha(d_\gamma)$ is invertible over $E$.
\end{proof}

In the proof of the foregoing result, we have used each $t_\beta$ to invert a matrix of $\Sigma$.
We remark that in \cite[Theorem~1]{kC03}, with only one $t$, an infinite countable number of matrices are inverted. We could have done something similar, but it would not be of any help because the set of $t$'s has to be of cardinality $\zeta$ since $X$ is of uncountable cardinality $\zeta$.

\medskip

Our next goal is to prove a result similar to
Corollary~\ref{cor:MNbig} for an uncountable set $X$.
For that we make a free abelian group into an ordered
group with a chain of principal convex subgroups which is order-isomorphic to a certain cardinal.

Let $\zeta$ be a cardinal, and  let $G_\zeta$ be the free abelian
group with a basis $\{e_\beta : \beta<\zeta\}$. Thus every
nonzero element $x$ of $G_\zeta$ is uniquely expressed as
\begin{equation}\label{eq:expressionelements}
  x=m_1e_{\beta_1}+\dots+m_re_{\beta_r},
\end{equation}
with $m_i\in\mathbb{Z}\setminus\{0\}$ and ordinals
$\beta_1<\dots<\beta_r<\zeta$.

We now make $G_\zeta$ into an ordered group in the following way. Given $x$ as in
\eqref{eq:expressionelements}, we say that $x>0$ if and
only if $m_r>0$. Then given $x,y\in G_\zeta$, it is not
difficult to see that the relation given by $x>y$ whenever $x-y>0$ is a total order in $G_{\zeta}$
and that $(G_{\zeta},<)$ is an ordered group.

For each ordinal $\alpha<\zeta$, let $H_\alpha$ be the subgroup
generated by $\{e_\beta\}_{\beta\leq\alpha}$. Note that $H_{\alpha}$ is a
convex subgroup of $(G_{\zeta},<)$. In fact, $H_\alpha$ is a principal convex
subgroup of $(G_{\zeta},<)$, because $H_\alpha=\{x\in G : -ne_\alpha\leq x\leq
ne_\alpha, \text{ for some natural number $n$}\}$.

\begin{corollary}
  Let $K$ be a skew field with centre $F$ such that the dimension of $K$ over
  $F$ is infinite.
  Let $X$ be a set of cardinality $\zeta$. Consider the ordered group
  $(G_\zeta,<)$. Then the field $K((G_\zeta;<))$ contains a free field
  $F(\langle X\rangle)$.
\end{corollary}

\begin{proof}
  We show that the conditions \eqref{conda} and \eqref{condb} of
  Theorem~\ref{theo:freefielduncountablerank} are satisfied.
  The chain of principal convex subgroups $\{H_\alpha :
  \alpha<\zeta\}$ of $G_\zeta$ is isomorphic, as an ordered set, to the
  cardinal $\zeta$.
  For each $\alpha<\zeta$, it can be proved that a series $f\in
  E_\alpha =K((H_\alpha;<))$ commutes with each $d\in K$ if and only if $f\in
  F_\alpha=F((H_\alpha;<))$. Thus $F_\alpha$ is the centre of
  $E_\alpha$. Similarly, the centre of
  $K((G_\zeta;<))$ is $F((G_\zeta;<))$.
  Now if $\{d_i\}_{i\in I}$ is a basis of $K$ as an $F$-vector space, then
  the set $\{d_i\}_{i\in I}$ is a subset of $E_\alpha$ which is linear independent
  over $F_\alpha$.
  Hence the dimension $[E_\alpha:F_\alpha]$ is infinite.
  It follows from Theorem~\ref{theo:freefielduncountablerank} that
  $K((G_\zeta;<))$ contains a free field $\ff{F((G_\zeta;<))}{}{X}$.
  In particular, it contains $\ff{F}{}{X}$.
\end{proof}


\begin{thebibliography}{99}

\bibitem{vB08}
  \textsc{V. V. Bavula},
  Finite generation of the group of eigenvalues for sets of derivations 
  or automorphisms of division algebras,
  \textit{Comm. Algebra}~\textbf{36} (2008), 2195--2201. 


\bibitem{BMR77}
  \textsc{R. Botto~Mura and A. Rhemtulla},
  \textit{Orderable groups},
  Marcel Dekker, New York, 1977.


\bibitem{kC03}
  \textsc{K. Chiba},
  Free fields in complete skew fields and their valuations,
  \textit{J. Algebra}~\textbf{263} (2003), 75--87.

\bibitem{pC77}
  \textsc{P. M. Cohn},
  \textit{Skew Field Constructions},
  Cambridge University Press, Cambridge, 1977.

\bibitem{pC85}
  \textsc{P. M. Cohn},
  \textit{Free Rings and Their Relations},
  2nd. Ed., Academic Press, London, 1985.

\bibitem{pC95}
  \textsc{P. M. Cohn},
  \textit{Skew Fields. Theory of General Division Rings},
  Cambridge University Press, Cambridge, 1995.

\bibitem{DL82}
  \textsc{W.~Dicks, and J.~Lewin},
  A Jacobian conjecture for free associative algebras,
  \textit{Comm. Algebra}~\textbf{10} (1982), 1285--1306.

\bibitem{gE06}
  \textsc{G. Elek},
  The amenability and non-amenability of skew fields,
  \textit{Proc. Amer. Math. Soc.}~\textbf{134} (2006), 637--644.

%

\bibitem{lF63}
  \textsc{L.~Fuchs},
  \textit{Partially Ordered Algebraic Systems},
  Pergamon Press, Oxford, 1963.

\bibitem{GSpp}
  \textsc{J. Z. Gon\c calves and M. Shirvani},
  A survey on free objects in division rings and in division rings with an involution,
  to appear in \textit{Comm. Algebra}.

\bibitem{iH70}
  \textsc{I. Hughes},
  Division rings of fractions for group rings,
  \textit{Comm. Pure Appl. Math.}~\textbf{23} (1970), 181--188.

\bibitem{sJ55}
  \textsc{S. A. Jennings},
  The group ring of a class of infinite nilpotent groups,
  \textit{Canad. J. Math.}~\textbf{7} (1955), 169--187.

\bibitem{tL01}
  \textsc{T. Y. Lam},
  \textit{A First Course in Noncommutative Rings},
  2nd. Ed., Springer-Verlag, New York, 2001.

\bibitem{tL99}
  \textsc{T. Y. Lam},
  \textit{Lectures on Modules and Rings},
  Springer-Verlag, New York, 1999.

\bibitem{jL74}
  \textsc{J. Lewin},
  Fields of fractions for group algebras of free groups,
  \textit{Trans. Amer. Math. Soc.}~\textbf{192} (1974), 339--346.

\bibitem{aL77}
  \textsc{A. I. Lichtman},
  On subgroups of the multiplicative group of skew fields,
  \textit{Proc. Amer. Math. Soc.}~\textbf{63} (1977), 15–-16.

\bibitem{aL95}
  \textsc{A. I. Lichtman},
  Valuation methods in division rings,
  \textit{J. Algebra}~\textbf{177} (1995), 870--898.

\bibitem{aL99}
  \textsc{A. I. Lichtman},
  Free subalgebras in division rings generated by universal
  enveloping algebras, \textit{Algebra Colloq.}~\textbf{6}
  (1999), 145--153.

\bibitem{lM83}
  \textsc{L. Makar-Limanov},
  The skew field of fractions of the Weyl algebra contains a
  free noncommutative subalgebra,
  \textit{Comm. Alg.}~\textbf{11} (1983), 2003--2006.

\bibitem{lM84}
  \textsc{L. Makar-Limanov},
  On group rings of nilpotent groups,
  \textit{Israel J. Math.}~\textbf{48} (1984), 244-248.

\bibitem{MM91}
  \textsc{L. Makar-Limanov and P. Malcolmson},
  Free subalgebras of enveloping fields,
  \textit{Proc. Amer. Math. Soc.}~\textbf{111} (1991), 315--322.


\bibitem{aM48}
  \textsc{A. I. Malcev},
  On the embedding of group algebras in division algebras,
  \textit{Doklady Akad. Nauk SSSR (N.S.)}~\textbf{60} (1948),
  1499--1501.

\bibitem{NR10}
  \textsc{A. Navas and C. Rivas},
  Describing all bi-orderings on Thompson's group $F$,
  \textit{Groups Geom. Dyn.}~\textbf{4} (2010), 163-–177.

\bibitem{bN49}
  \textsc{B. H. Neumann},
  On ordered division rings,
  \textit{Trans. Amer. Math. Soc.}~\textbf{66} (1949), 202--252.

\bibitem{dP89}
  \textsc{D. S. Passmann},
  \textit{Infinite Crossed Products},
  Academic Press, Boston, MA, 1989.


\bibitem{mS80}
  \textsc{M. E. Sweedler}, 
  Tensor products of division rings and finite generation of subdivision rings,
  \textit{Bull. Inst. Math. Acad. Sinica}~\textbf{8} (1980), 385--387. 

\end{thebibliography}
\end{document}